\documentclass{amsart}

\usepackage{amsmath,amssymb,amsfonts,amssymb,amsthm}

\usepackage{verbatim}
\usepackage[usenames]{color}
\usepackage{hyperref}
\usepackage{url}
\usepackage{tikz,tikz-qtree,cancel}
\usepackage{array,tikz-qtree,cancel}
\usepackage{graphicx}
\usepackage{adjustbox}
\usepackage{amsthm,graphicx,tikz,appendix,tikz-qtree,ifthen,cancel}
\usetikzlibrary{calc,shapes,patterns,positioning}
\usepackage{stmaryrd}

\usetikzlibrary{matrix}

\newtheorem{thm}{Theorem}[section]

\newtheorem{lem}[thm]{Lemma}
\newtheorem{cor}[thm]{Corollary}

\newtheorem*{thmHL}{Theorem \ref{thm.HL1theta}}
\newtheorem*{thm1.5}{Theorem \ref{thm.1beta2}}
\newtheorem*{thm1.6}{Theorem \ref{nsbpr->negHL}}
\newtheorem*{thm1.7}{Theorem \ref{thm.DiamondimpliesNotHL}}
\newtheorem*{cor1.8}{Theorem \ref{cor.L}}

\newtheorem{claim}{Claim}
\newtheorem{fact}[thm]{Fact}

\theoremstyle{remark}
\newtheorem{rem}[thm]{Remark}

\theoremstyle{definition}
\newtheorem{defn}[thm]{Definition}

\newtheorem{notation}[thm]{Notation}

\newtheorem{question}[thm]{Question}

\theoremstyle{remark}

\newcommand{\utilde}{\widetilde}

\newcommand{\al}{\alpha}
\newcommand{\om}{\omega}

\newcommand{\ve}{\varepsilon}
\newcommand{\sse}{\subseteq}

\newcommand{\forces}{\Vdash}
\newcommand{\bm}{\mathbf{m}}
\newcommand{\bn}{\mathbf{n}}

\DeclareMathOperator{\ran}{ran}
\DeclareMathOperator{\dom}{dom}

\DeclareMathOperator{\stem}{stem}

\DeclareMathOperator{\lh}{lg}
\DeclareMathOperator{\ot}{o.t.}
\DeclareMathOperator{\cf}{cf}
\DeclareMathOperator{\PR}{Pr}
\DeclareMathOperator{\HL}{HL}
\DeclareMathOperator{\OBep}{{OB}_{\varepsilon}}
\DeclareMathOperator{\OB}{OB}

\newcommand{\re}{\restriction}
\newcommand{\bP}{\mathbb{P}}

\newcommand{\bQ}{\mathbb{Q}}

\newcommand{\ra}{\rightarrow}

\newcommand{\Lra}{\Longrightarrow}

\newcommand{\Llra}{\Longleftrightarrow}
\newcommand{\lgl}{\langle}
\newcommand{\rgl}{\rangle}

\newcommand{\kgtwo}{{}^{\kappa>}2}
\newcommand{\naltwo}{{}^n({}^{\al}2)}
\newcommand{\trngl}{\trianglelefteqslant}
\newcommand{\utcD}{\utilde{\mathcal{D}}}

\newcommand{\Erdos}{Erd{\H{o}}s}

\newcommand{\Lauchli}{L{\"{a}}uchli}

\newcommand{\noprint}[1]{\relax}

\begin{document}

\title[Halpern--L\"{a}uchli at singular  cardinals]{The Halpern--L\"{a}uchli Theorem at singular cardinals and failures of weak versions}

\author{Natasha Dobrinen}
\address{University of  Notre Dame\\
Department of Mathematics, 
255 Hurley Bldg, Notre Dame, IN 46556 USA}
\email{ndobrine@nd.edu}

 \thanks{This work was partially supported by 
  National Science Foundation Grant  DMS 1833363.  This is paper 1230 on Shelah's list.}

\author{Saharon Shelah}
\address{The Hebrew University of Jerusalem\\
Einstein Institute of Mathematics, 
Edmond J.\ Safra Campus, Givat Ram, Jerusalem, 91904, Israel, and 
Department of Mathematics, Hill Center - Busch Campus, Rutgers, The State University of New Jersey, 110 Frelinghuysen Road, Piscataway, NJ 08854 USA}
\email{shelah@math.huji.ac.il}



\maketitle

\begin{abstract}
This paper continues a line of investigation  of  the Halpern--\Lauchli\ Theorem at uncountable cardinals. 
We prove in ZFC that the Halpern--L\"{a}uchli Theorem for one tree of height $\kappa$ holds whenever $\kappa$ is 
 strongly inaccessible  and the coloring takes less  than $\kappa$ colors. 
We  prove consistency of  the  Halpern--L\"{a}uchli Theorem
for  finitely many  trees of  height $\kappa$, 
where $\kappa$  is
 a strong limit cardinal of countable cofinality.
On the other hand, we
prove  failure of   weak forms of 
Halpern--\Lauchli\  for trees of height $\kappa$, whenever 
 $\kappa$  is a
  strongly 
 inaccessible, non-Mahlo cardinal or  a
 singular strong limit cardinal with cofinality the successor of a regular cardinal. 
We also prove
 failure in $L$ of  a weak version for all  strongly inaccessible, non-weakly compact cardinals.
\end{abstract}


\section{Introduction}

Investigations of the Halpern--\Lauchli\ Theorem on trees of uncountable height commenced with work of the second author in \cite{Shelah91}.
In that paper, Shelah
built on a forcing proof due to  Harrington  for trees of height $\om$ to
show the consistency of a strong version of the  Halpern--\Lauchli\ Theorem for 
 trees of height $\kappa$, where $\kappa$ is measurable in certain forcing extensions. 
A slightly modified version of this theorem was applied by D\v{z}amonja, Larson, and Mitchell 
to characterize  the big Ramsey degrees for the $\kappa$-rationals 
in \cite{Dzamonja/Larson/MitchellQ09}
and the $\kappa$-Rado graph in 
\cite{Dzamonja/Larson/MitchellRado09}, for such  $\kappa$. 
More recently, consistency strengths of various versions of the Halpern--\Lauchli\ Theorem at uncountable cardinals  were investigated in \cite{Dobrinen/Hathaway16}, \cite{Dobrinen/Hathaway18}, and 
\cite{Zhang17}.  
This 
line of investigation is continued in this article.

Let $\kappa$ be an ordinal.
For nodes $\eta,\nu\in  {}^{\kappa>}2$,
we  write
$\eta\trngl\nu$   when 
 $\eta$ is an initial segment of $\nu$,
and  write  $\eta\lhd \nu$  when $\eta$ is a proper initial segment of $\nu$. 
The {\em length} of $\eta$, 
denoted by 
$\lg(\eta)$, 
is  the ordinal $\al$ such that $\eta\in {}^{\al}2$.
A subset $T\sse {}^{\kappa>}2$ is a {\em subtree} if $T$ is non-empty and  closed under initial segments.
Similarly to \cite{Dobrinen/Hathaway16}, 
 we call a subtree $T\sse {}^{\kappa>}2$   {\em regular} if  the following hold:
\begin{enumerate}
\item
For all $\eta\in T$ and  $\al<\kappa$, there is a $\nu\in T$ such that $\eta\trngl \nu$ and $\lg(\nu)=\max\{\lg(\eta),\al\}$;
\item
If
$\delta<\kappa$ is a limit ordinal and
$\eta\in{}^{\delta}2$ has the property that  $\eta\re\al\in T$ for all $\al<\delta$, 
then $\eta\in T$.
\end{enumerate}

The following is the  strong-tree version of the Halpern--\Lauchli\ Theorem  for finitely many trees of  height $\kappa$.

\begin{defn}\label{def.HL}
\begin{enumerate}
\item
(HL$_{n,\theta}(\kappa)$)
For finite $1\le n<\om$ and $2\le\theta$, we write  HL$_{n,\theta}(\kappa)$  to
denote that 
for any $c:\bigcup\{ {}^n({}^{\al}2):\al<\kappa\}\ra \theta$,
there are  $A, T_0,\dots,T_{n-1}$ satisfying  the following:
\begin{enumerate}
\item[(a)]
$A\in[\kappa]^{\kappa}$.
\item[(b)]
$T_{\ell}$ is a regular subtree of $\kgtwo$, for each $\ell<n$.
\item[(c)]
If $\eta\in T_{\ell}$, then
$\{\eta^{\frown}\lgl 0\rgl,\eta^{\frown}\lgl 1\rgl\}\sse T_\ell$ iff  $\lg(\eta)\in A$.
\item[(d)]
$c\re (\bigcup\{\prod_{\ell<n} T_{\ell|\varepsilon}:\varepsilon\in A\})$ is constant, where $T_{\ell|\varepsilon}=\{\eta\in T_\ell:\lg(\eta)=\varepsilon\}$.
\end{enumerate}
\item
A tree $T\sse \kgtwo$ is called a {\em strong tree} if there is an $A\in[\kappa]^{\kappa}$ such that (a)--(c)  hold for $T$.
\end{enumerate}
\end{defn}

Variations of  the Halpern--\Lauchli\  Theorem will also be investigated in this paper.  For their statements, the following notion of suitable triple will be useful.

\begin{defn}\label{defn.suitabletriple}
\begin{enumerate}
\item
A triple $(\kappa,T,A)$ is called  {\em suitable} if the following hold:
\begin{enumerate}
\item[(a)]
 $A\in[\kappa]^{\kappa}$;
 \item[(b)]
 $T$ is a regular subtree of ${}^{\kappa>}2$;
\item[(c)]
If $\eta\in T\cap{}^{\al}2$, then
$\{\eta^{\frown}0,\eta^{\frown}1  \}\sse T$
iff $\al\in A$.
 \end{enumerate}
 \item
 Given $\ell<\om$,  let $(0_{\ell})$ denote the sequence of $0$'s of length $\ell$, and let 
 $\eta^*_\ell=(0_{\ell})^{\frown}1$, the sequence  of length $\ell+1$ where the first $\ell$ coordinates are $0$ and the last coordinate is $1$. 
 \end{enumerate}
\end{defn}

Note that $(\kappa,T,A)$ is a suitable triple  if and only if $T$ is a strong tree with 
$A$ being the set of lengths of nodes in $T$ which branch.



\begin{defn}[Halpern--\Lauchli\ Variations]\label{defn.HLvar}
Let  $1\le n<\om$ and $2\le \theta<\kappa$ be given,  with $\kappa$ strongly inaccessible.
\begin{enumerate}
\item
$\HL_{n,\theta}[\kappa]$ abbreviates the following statement:
Given a coloring
$c:\bigcup\{{}^n({}^{\al}2):\al<\kappa\}\ra \theta$,
then there is a suitable triple $(\kappa,T,A)$ and a color $\theta_*<\theta$ such that
for all
$\al\in A$ and 
$\overline{\nu}=(\nu_0,\dots\nu_{n-1})$ with 
$\eta_\ell^*\trngl \nu_\ell\in T\cap {}^{\al}2$
for each $\ell<n$,
then
$c(\overline{\nu})\ne \theta_*$.

\item
HL$^+_{n,\theta}(\kappa)$ abbreviates the following statement:
Given $\overline{<}=\lgl <_{\al}:\al<\kappa\rgl$, where  $<_{\al}$ is a well-ordering of ${}^{\al}2$,
and given a coloring $c:\bigcup\{[{}^{\al}2]^n:\al<\kappa\}\ra \theta$,
there  is a suitable triple $(\kappa,T,A)$
so that
whenever $\al<\kappa$ 
and $\eta_0,\dots,\eta_{n-1}\in T\cap {}^{\al}2$ are pairwise distinct,
then $c$ is constant on the set
\begin{align*}
\{\overline{\nu}=(\nu_0,\dots\nu_{n-1}):\bigwedge_{\ell<n}\eta_{\ell}\triangleleft \nu_\ell\in T,\ 
&\lg(\nu_0)=\dots=\lg(\nu_{n-1})\in A,\cr
& \ \mathrm{\ and\ }\bar{\nu}\mathrm{\ is\ }
<_{\lg(\nu_0)}\mathrm{-increasing}\}.
\end{align*}

\item
HL$^+_{n,\theta}[\kappa]$ abbreviates the following statement:
Given $\overline{<}=\lgl <_{\al}:\al<\kappa\rgl$, where $<_{\al}$ is a well-ordering of ${}^{\al}2$,
and given a coloring $c:\bigcup\{[{}^{\al}2]^n:\al<\kappa\}\ra \theta$,
there is a suitable triple $(\kappa,T,A)$
so that
whenever $\al<\kappa$  
and $\eta_0,\dots,\eta_{n-1}\in T\cap {}^{\al}2$ are pairwise distinct,
then
$c$ misses at least one color on the set
\begin{align*}
\{\overline{\nu}=(\nu_0,\dots\nu_{n-1}):\bigwedge_{\ell<n}\eta_{\ell}\triangleleft \nu_\ell\in T,\ 
&\lg(\nu_0)=\dots=\lg(\nu_{n-1})\in A, \cr
& \  \mathrm{\ and\ }\bar{\nu}\mathrm{\ is\ }
<_{\lg(\nu_0)}\mathrm{-increasing}\}.
\end{align*}
\end{enumerate}
\end{defn}


We point out the following straightforward facts.

\begin{fact}\label{fact.simple}
\begin{enumerate}
\item
For any given triple $n,\theta,\kappa$,
the following implications hold:
\begin{center}
\begin{tikzpicture}
  \matrix (m) [matrix of math nodes,row sep=3em,column sep=4em,minimum width=2em]
  {
     \HL^+_{n,\theta}(\kappa) & \HL_{n,\theta}(\kappa)\\
     \HL^+_{n,\theta}[\kappa] &  \HL_{n,\theta}[\kappa]\\};
  \path[-stealth]
    (m-1-1) edge (m-2-1)
            edge   (m-1-2)
    (m-2-1.east|-m-2-2) edge 
             (m-2-2)
    (m-1-2) edge  (m-2-2);
\end{tikzpicture}
\end{center}
\item
For $m\le n$ and $2\le \theta\le \theta'$, 
the following implications  hold:
\begin{align*}
\HL_{n,\theta'}(\kappa) &\ \longrightarrow\ \HL_{m,\theta}(\kappa)\cr
\HL_{n,\theta'}^+(\kappa) &\ \longrightarrow\ \HL^+_{m,\theta}(\kappa)\cr
\HL_{n,\theta}^+[\kappa] &\ \longrightarrow\ \HL^+_{m,\theta'}[\kappa]\cr
\HL_{n,\theta}[\kappa] &\ \longrightarrow\ \HL^+_{m,\theta'}[\kappa]
\end{align*}
\item  

The following versions are equal:
\begin{align*}
\HL_{n,2}(\kappa)&=\HL_{n,2}[\kappa]\cr
\HL^+_{n,2}(\kappa)&=\HL^+_{n,2}[\kappa]
\end{align*}
\end{enumerate}
\end{fact}

We briefly review some highlights from previous work. 
Shelah proved in \cite{Shelah91} that 
$\HL^+_{n,\theta}(\kappa)$ holds
for all $1\le n<\om$ and $2\le \theta<\kappa$
whenever $\kappa$ is a cardinal  
with the following property $(*)$:  $\kappa$ 
 is measurable  after forcing with Cohen$(\kappa,\lambda)$, where  
$\lambda\ra(\kappa^+)^{2n}_{2^{\kappa}}$.
If $\kappa$ is a $\kappa+2n$-strong cardinal,
 then $(*)$ is satisfied by 
$\lambda=(\beth_{2n}(\kappa))^+$.
Utilizing  a lemma from \cite{Shelah91}, Zhang proved in  \cite{Zhang17}
a ``tail-cone'' version which is  intermediate between $\HL^+_{n,\theta}(\kappa)$ and $\HL_{n,\theta}(\kappa)$.
He  then  applied the tail-cone version to obtain a polarized partition relation for finite products of $\kappa$-rationals, for $\kappa$ satisfying $(*)$, 
proving  an analogue of Laver's result for finite products of rationals in \cite{Laver84}.

In 
\cite{Dobrinen/Hathaway16}, the first author and Hathaway   proved that $\HL_{1,n}(\kappa)$ holds for any $2\le n<\om$ when  $\kappa$ is  strongly inaccessible.
Soon after,
Zhang  showed in \cite{Zhang17} that  when $\kappa$ is weakly compact, then 
$\HL_{1,\theta}(\kappa)$ holds for  all $2\le\theta<\kappa$.
We will improve both results by proving the following:

\begin{thm}\label{thm.HL1theta}
If $\kappa$ is strongly inaccessible and $2\le \theta<\kappa$,
then $\HL_{1,\theta}(\kappa)$ holds.
\end{thm}

In  \cite{Dobrinen/Hathaway16}, the upper bound for the  consistency strength of 
 $\HL_{n,\theta}(\kappa)$,
for  $\kappa$ strongly inaccessible, $2\le n<\om$, and 
$2\le \theta<\kappa$,
was  reduced  from  a $\kappa+2n$-strong cardinal to a
 $\kappa+n$-strong cardinal. 
Our first theorem
extends this result to 
strong limit cardinals $\kappa$ of countable cofinality.
The hypotheses of Theorem \ref{thm.1beta2} are satisfied whenever $\kappa$ is  a $\kappa+n$-strong cardinal.

\begin{thm}\label{thm.1beta2}
Let
$1\le \bn<\om$ and $2\le k<\om$ be given.
Suppose
$\lambda\ge (\beth_{\bn}(\kappa))^+$ and that   $\kappa$ is measurable in the generic extension  via
$\bP=$
Cohen$(\kappa,\lambda)$  forcing.
 Let $\utilde{\bQ}$ be a $\bP$-name for Prikry forcing.
Then
$\HL_{n,k}(\kappa)$ holds for all  $1\le n\le \bn$ 
in the
generic extension  forced by $\bP * \utilde{\bQ}$.
\end{thm}

D\v{z}amonja, Larson, and Mitchell  pointed out in \cite{Dzamonja/Larson/MitchellQ09} that $\HL^+_{2,2}(\kappa)$ implies that $\kappa$ is weakly compact.
In the following theorem
we find lower bounds for the weak version $\HL_{2,\theta}[\kappa]$, for all $2\le \theta<\kappa$.

\begin{thm}\label{nsbpr->negHL}
\begin{enumerate}
\item
If $\kappa$ is the first inaccessible, then $\HL^+_{2,\theta}[\kappa]$ fails, for each $2\le \theta<\kappa$.

\item
Suppose  $\kappa$ is inaccessible 
and not Mahlo.
Then  $\HL^+_{2,\theta}[\kappa]$ fails, for each $2\le \theta<\kappa$.

\item
If $\kappa$ is a singular  strong limit cardinal  and
 $\cf(\kappa)=\mu^+$, 
with $\mu$ regular,  
then
 $\HL^+_{2,\mu^+}[\kappa]$ fails.
\end{enumerate}
\end{thm}

\begin{thm}\label{thm.DiamondimpliesNotHL}
Assume $\kappa$ is strongly inaccessible,  $S\sse \kappa$ is a non-reflecting stationary set,
and $\diamondsuit_S$ holds.
Then $\HL_{2,\theta}[\kappa]$ fails, for each $2\le \theta\le\kappa$.
\end{thm}

Since the hypotheses of the previous theorem hold in $L$  for all strongly inaccessible $\kappa$ which are not weakly compact, we have the following corollary.

\begin{cor}\label{cor.L}
If $V=L$, then 
for all strongly inaccessible, non-weakly compact $\kappa$ and for 
each $2\le \theta\le\kappa$,
$\HL_{2,\theta}[\kappa]$ fails.
\end{cor}


\section{Halpern--\Lauchli\   on one tree}\label{sec.Sec1}

In 
\cite{Dobrinen/Hathaway16},
Hathaway and the first author proved  that $\HL_{1,n}(\kappa)$ holds for all  weakly compact
 $\kappa$ and 
all $2\le n<\om$;
Zhang pointed out that the  proof in \cite{Dobrinen/Hathaway16}   actually  implies $\HL_{1,n}(\kappa)$ holds  for all strongly inaccessible cardinals $\kappa$.
In \cite{Zhang17},
Zhang proved 
$\HL_{1,\theta}(\kappa)$ holds  for all weakly compact 
 $\kappa$  and all $2\le\theta<\kappa$.
Zhang also proved two  consistency results 
showing that under certain large cardinal assumptions, 
it is consistent that there is a strongly inaccessible, not weakly compact cardinal $\kappa$ such that 
 for all $2\le\theta<\kappa$, $\HL_{1,\theta}(\kappa)$ holds. 
(See Corollary 5.7 and Theorem 5.8 in \cite{Zhang17}.)

The following theorem shows that the strong tree version of  Halpern--\Lauchli\  holds on one tree for all strongly inaccessible $\kappa$ and all colorings into less than $\kappa$ many colors.

\begin{thmHL}
If $\kappa$ is strongly inaccessible and $2\le \theta<\kappa$,
then $\HL_{1,\theta}(\kappa)$ holds.
\end{thmHL}

\begin{proof}
Suppose that $\kappa$ is strongly inaccessible and that $2\le \theta<\kappa$ is the least ordinal such that 
 $\HL_{1,\theta}(\kappa)$ fails.
By a result in \cite{Dobrinen/Hathaway16},
 $\theta$ must be at least $\om$;
furthermore, it is straightforward to see that 
$\theta$ must be a regular cardinal.
Let $c:{}^{\kappa>}2\ra\theta$ be a coloring which witnesses failure of  $\HL_{1,\theta}(\kappa)$.
Without loss of generality, we may 
 assume that  $\theta$ is the range of the coloring $c$.

For $\varepsilon<\theta$, 
define $\OBep$ to be the set of triples 
\begin{equation}
\bm=(T,A,\al)=(T_{\bm},A_{\bm},\al_{\bm})
\end{equation}
satisfying the following (a)--(f):
\begin{enumerate}
\item[(a)]
$\al<\kappa$ and $A\sse\al$ is unbounded in $\al$.
\item[(b)]
$T$ is a subtree of ${}^{\al\ge}2$.
\item[(c)]
If $\eta\in T$, then there is a $\nu\in T\cap{}^{\al}2$ such that $\eta\trngl\nu$.
\item[(d)]
If $\eta\in T$, then (${\eta}^{\frown}0$ and 
${\eta}^{\frown}1$ are both in $T$) $\Longleftrightarrow$ $\lg(\eta)\in A$.
\item[(e)]
If $\delta\le\al$ is a limit ordinal and $\eta\in {}^{\delta}2$, then 
$\eta\in T$  $\Longleftrightarrow$ 
$\forall \beta<\delta\, (\eta\re\beta\in T)$.
\item[(f)]
$c$ is constant with value $\ve$ on $T\cap(\bigcup_{\beta\in A}{}^{\beta}2)$.
\end{enumerate}

Define $\le_{\ve}$ as the following $2$-place relation on $\OBep$:
For $\bm,\mathbf{n}\in \OBep$,
$\bm\le_{\ve}\mathbf{n}$ if and only if 
$\al_{\bm}\le \al_{\mathbf{n}}$,
 $A_{\bm}=A_{\mathbf{n}}\cap\al_{\bm}$,
and 
$T_{\bm}=T_{\mathbf{n}}\cap {}^{\al_{\bm\ge}}2$.

The following facts are straightforward.

\begin{fact}\label{fact.4}
\begin{enumerate}
\item
$\le_{\ve}$ is a partial order on $\OBep$.

\item
If $\lgl \bm_i:i<\delta\rgl$ is a $\le_{\ve}$-increasing sequences, where $\delta<\kappa$,
then the sequence has a $\le_{\ve}$-least upper bound.
\end{enumerate}
\end{fact}

\begin{proof}
(1) is clear.  For (2), take  $\al=\sup_{i<\delta}\al_{\bm_i}$ and  $A=\bigcup_{i<\delta}A_{\bm_i}$, and take 
$T$ to be $\bigcup_{i<\delta}T_{\bm_i}$ along with all maximal branches in $\bigcup_{i<\delta}T_{\bm_i}$.
Then $\bm=(\al,A,T)$ is a member of $\OBep$ and is the 
$\le_{\ve}$-least upper bound of $\lgl \bm_i:i<\delta\rgl$.
\end{proof}


\begin{lem}\label{lem.5}
For each $\ve<\theta$ and each $\bm\in\OBep$, there is an $\bn\in\OBep$ such that  
$\bm \le_{\ve}\bn$ and $\bn$ is 
$\le_{\ve}$-maximal.
\end{lem}

\begin{proof}
Suppose not.  
Then there are $\ve<\theta$ and $\bm_0\in\OBep$ such that for each 
$\bn\in\OBep$, $\bm_0\le_{\ve}\bn$ implies that $\bn$ is not maximal.
Thus, we can build a $\le_{\ve}$-strictly increasing sequence $\lgl \bm_i:i<\kappa\rgl$ as follows:
Given $\bm_i$,  since $\bm_0\le_{\ve}\bm_i$
there is some $\bm_{i+1}\in 
\OBep$ such that 
$\bm_i<_{\ve}\bm_{i+1}$.
If $i<\kappa$ is a limit ordinal, take $\bm_i$ to be the least upper bound of $\lgl \bm_j:j<i\rgl$, guaranteed by Fact \ref{fact.4}.

Let $A=\bigcup_{i<\kappa}$ and $T=\bigcup_{i<\kappa}T_{\bm_i}$.
Note that $A\in [\kappa]^{\kappa}$ 
since $\sup_{i<\kappa}\al_{\bm_i}=\kappa$ and each
$A_{\bm_i}$  is unbounded in $\al_{\bm_i}$.
Thus, $(\kappa,T,A)$ is a suitable triple.   
But then $c$ has constant value $\ve$ on 
$\bigcup_{\al\in A}T\cap {}^{\al}2$, contradicting that $c$ witnesses the failure of $\HL_{1,\theta}(\kappa)$.
\end{proof}

We will choose $(\Lambda_i, \overline{\bm}_i, \overline{\ve}_i)$ by induction on $i<\theta$ satisfying the following.

\begin{enumerate}
\item[(a)]
$\Lambda_i$ is a nonempty set of  pairwise $\lhd$-incomparable nodes in ${}^{\kappa>}2$.

\item[(b)]
If $j<i$, then for each $\eta\in\Lambda_i$ there is a unique $\nu\in \Lambda_j$ such that $\nu\lhd \eta$.
(It follows that $\Lambda_j\cap\Lambda_i=\emptyset$.)

\item[(c)]
$\overline{\ve}_i=\lgl \ve_{\eta}=\ve(\eta):\eta\in\Lambda_i\rgl$,
where $\ve_{\eta}$ is the minimum ordinal 
 in  the range of $c$ on $\{\zeta\in {}^{\kappa>}2:\eta\trngl \zeta\}$ 
which is also above 
$\sup\{\ve_{\nu}:\exists j<i\, (\nu\in\Lambda_j)\wedge\nu\lhd \eta\}$.

\item[(d)]
$\overline{\bm}_i=\lgl\bm_{\eta}=\bm(\eta):\eta\in\Lambda_i\rgl$.
Notation: $\bm_\eta=(T_{\eta}, A_{\eta}, \al_{\eta}=\al(\eta))$.
(There is no ambiguity using $\eta$ as an index since the $\Lambda_i$ will be disjoint.)

\item[(e)]
$\bm_{\eta}$ is a $\le_{\ve(\eta)}$-maximal   member of $\OB_{\ve(\eta)}$.

\item[(f)]
$\eta\in T_{\eta}$ and $\lg(\eta)\le \min(A_{\eta})$.
(This implies that $\eta\trngl\stem(T_\eta)$.)

\item[(g)]
If $i$ is a limit ordinal and
$\lgl \eta_j:j<i\rgl$, 
 $\eta_j\in \Lambda_j$,
 is an $\lhd$-increasing sequence,
then $\bigcup_{j<i}\eta_j\in\Lambda_i$.
\end{enumerate}

We now carry out the inductive construction.
\vskip.1in

\noindent {\bf Case 1:}  $i=0$.
Let $\Lambda_0=\{\lgl\rgl\}$ and  $\ve_{\lgl\rgl}=0$.
Take $\bm_{\lgl\rgl}$  to be any $\le_0$-maximal member of $\OB_0$, and let $\overline{\bm}_0=\lgl \bm_{\lgl\rgl}\rgl$.
\vskip.1in

For Cases 2 and 3, we use the following notation.
For  $\eta\in {}^{\kappa>}2$,
define 
\begin{equation}
\Theta_\eta=\{\ve\in\theta:\exists\zeta\, (\eta\trngl\zeta\wedge c(\zeta)=\ve)\}.
\end{equation}
That is, $\Theta_\eta$ is the range of $c$ on the set of  nodes in ${}^{\kappa>}2$ extending $\eta$.
Note that $|\Theta_\eta|=\theta$, since $\theta$ is by assumption the least ordinal for which $\HL_{1,\theta}(\kappa)$ fails and the coloring $c$  witnesses this failure.
\vskip.1in

\noindent {\bf Case 2:}  $i=j+1$.
Let $\Lambda_i=\bigcup_{\nu\in\Lambda_j} T_\nu\cap {}^{\al(\nu)}2$.
Since $\Lambda_j$ is an antichain, so is $\Lambda_i$.
 By (f) of the induction hypothesis,
for each $\nu\in \Lambda_j$,
 $\nu\lhd\stem(T_\nu)$ and hence  $|\nu|<\al(\nu)$.
Given $\eta\in\Lambda_i$,
choose
\begin{equation}
\ve_{\eta}=
\min(\Theta_\eta\setminus
\{\ve_\nu:\nu\lhd\eta\mathrm{\ and\ }\exists k<i\, (\nu\in\Lambda_{k})\}).
\end{equation}
Then choose  $\bm_\eta$  to be some $\le_{\ve_{\eta}}$-maximal member of $\OB_{\ve_\eta}$ such that 
$\eta\trngl \stem(T_{\eta})$.
\vskip.1in

\noindent {\bf Case 3:}  $i$ is a limit ordinal.
Let $\Lambda_i$ be the set of  all nodes $\eta\in T$ such that 
$\lg(\eta)$ is a limit ordinal and $\forall\, \al<\lg(\eta)$,  $\exists\, j<i\ \exists\, \nu\in \Lambda_j\, (\nu\lhd\eta\wedge \lg(\nu)\ge\al)$.
In other words, $\Lambda_i$ is the set of limits of  $\lhd$-increasing sequences 
$\lgl \nu_j:j<i\rgl$  with each  $\nu_j\in \Lambda_j$.
For $\eta\in\Lambda_i$,
 choose $ \ve_\eta$ and $\bm_\eta$ as in Case 2.
\vskip.1in

Let 
$\Lambda=\bigcup_{i<\theta}\Lambda_i$.
Note that $|\Lambda|<\kappa$
since  $|\Lambda_i|<\kappa$ for each $i<\theta$, 
and  $\kappa$ is  regular.
Fix some $\al(*)<\kappa$  greater than $\sup\{\al_\eta:\eta\in \Lambda\}$.
We choose  by induction 
a sequence $\lgl \eta_i:i<\theta\rgl$ 
such that 
\begin{enumerate}
\item[(a)]
$\eta_i\in\Lambda_i$;
\item[(b)]
$j<i$ implies $\eta_j\lhd\eta_i$;
\item[(c)]
 $\eta_i\trngl\nu\in {}^{\al(*)}2$ implies  $c(\nu)>\ve_{\eta_j}$ for all $j<i$.
\end{enumerate}

\noindent{\bf Case 1:} $i=0$.
Let $\eta_0\in \Lambda_0$; that is, $\eta_0=\lgl\rgl$.
\vskip.1in

\noindent{\bf Case 2:} $i$ is a limit ordinal.
By   the construction of the $(\Lambda_i,\overline{\bm}_i,\overline{\ve}_i)$,
 $\eta_i=\bigcup\{\eta_j:j<i\}$ belongs to $\Lambda_i$,
where $\eta_j$ is the member of $\Lambda_j$ such that $\eta_j\lhd \eta_i$.
Clearly,  (b)  holds, and  (c) follows from 
(b) and (c) holding for all $j<i$.
\vskip.1in

\noindent{\bf Case 3:} $i=j+1$.
Then $\eta_j\in\Lambda_j$.
Let 
\begin{equation}
\Omega=\{\eta\in T_{\eta_j}\cap{}^{\al(\eta_j)}2: \exists \nu\in {}^{\al(*)}2\,
 (\eta\lhd\nu\wedge c(\nu)=\ve_{\eta_j})\}.
\end{equation}
Now, if $\Omega=T_{\eta_j}\cap{}^{\al(\eta_j)}2$ then we get a contradiction to $\bm(\eta_j)$ being  $\le_{\ve(\eta_j)}$-maximal.
So we can choose $\eta_i\in T_{\eta_j}\cap {}^{\al(\eta_j)}2$ which is not in $\Omega$.
Then $\eta_i$ is in $\Lambda_i$, so (a) holds, and 
for all $\nu\in {}^{\al(*)}2$ such that $\eta_i\lhd \nu$, $c(\nu)\ne \ve_{\eta_j}$.

Recall that by  the definition of 
 $\ve_{\eta_i}$,
for each $\nu$ above $\eta_i$, either $c(\nu)\ge\ve_{\eta_i}$ or else $c(\nu)= \ve_{\eta_{k}}$ for some 
 $k\le j$.
We have already seen that 
$c(\nu)\ne \ve_{\eta_j}$, and 
 by the induction hypothesis,
$c(\nu)> \ve_{\eta_{k}}$ for all $k<j$.
Thus, (c) holds.
Note that (b) holds since
 $\eta_j\lhd\stem(T_{\eta_j})$.

This finishes the construction of a sequence $\lgl \eta_i:i<\theta\rgl$ satisfying (a)--(c).
Let $\eta=\bigcup_{i<\theta} \eta_i$,
noting that $\lg(\eta)\le\al(*)$.
Take any $\nu\in {}^{\al(*)}2$ such that $\eta\trngl \nu$.
Then 
for  each  $i<\theta$, $\eta_{i+1}\trngl\nu$  so 
(c)  implies  that 
 $c(\nu)>\ve_{\eta_{i}}$.
The  sequence of  ordinals
$\lgl \ve_{\eta_i}:i<\theta\rgl$
 is  strictly increasing, 
so
  $\sup_{i<\theta}\ve_{\eta_i}=\theta$
since $\theta$ is regular, implying that $c(\nu)\ge\theta$.
But this  contradicts  that $c(\nu)$ must be in $\theta$.
\end{proof}


\section{Halpern--\Lauchli\  at singular cardinals of countable cofinality}\label{sec.Sec1}

In this section, we prove
Theorem \ref{thm.1beta2}, the consistency of  
 $\HL_{n,\theta}(\kappa)$ 
for 
$\kappa$  a singular cardinal of countable cofinality.

\begin{notation}
Given  $1\le n<\om$, we  define the function $\xi$ on  $\bigcup\{{}^{n}({}^{\al}2):\al<\kappa\}$ as follows:
For $\al<\kappa$ and  $\bar{\eta}\in {}^{n}({}^{\al}2)$, let
\begin{equation}
\xi(\bar{\eta})=\min\{\beta\le \al+1: \beta\le \al\ \mathrm{implies\ } \lgl \eta_{\ell}\re \beta:\ell<n\rgl\mathrm{\ has \ no \ repetition}\}.
\end{equation}
Thus,
for $\al<\kappa$ and $\bar{\eta}\in {}^n({}^{\al}2)$,
 if all members of the tuple $\bar{\eta}$ are distinct,
then $\xi(\bar{\eta})$ is the least ordinal where they are all distinct;
if the members of $\bar{\eta}$ are not all distinct,
then $\xi(\bar{\eta})=\al+1$.

Recall that given $\ell<\om$, we let $\eta^*_\ell=\lgl (0_{\ell}), 1\rgl$ denote the sequence  of length $\ell+1$ where the last entry is $1$ and all other entries are $0$.
Given $1\le n<\om$,
define
\begin{equation}\label{eq.A_n}
A_n(\kappa)=\{\bar\eta=(\eta_0,\dots,\eta_{n-1})\in \bigcup_{\al<\kappa}\naltwo:
 (\forall \ell<n)\, \eta_{\ell}^*\trngl \eta_\ell\},
\end{equation}
and for $m<\om$, define
\begin{equation}\label{eq.A_nm}
A_{n,m}(\kappa)=\{(u,\bar{\eta}):\exists \gamma<\kappa\, (u\in [\gamma]^{m}\ \mathrm{ and\ } \bar{\eta}\in {}^{n}({}^{\gamma}2))\}.
\end{equation}
When $\kappa$ is clear, we omit it and simply write $A_n$ and $A_{n,m}$.
\end{notation}


\begin{lem}\label{lem.e8.AimpliesB}
(A) $\Longrightarrow$ (B), where

(A) is the statement:
\begin{enumerate}
\item[(a)]
$1\le \bn$, $\kappa$ is inaccessible,
and
$\lambda\ge(\beth_{\mathbf{n}}(\kappa))^+$.
\item[(b)]
$\bP$ is Cohen$(\kappa,\lambda)$.
\item[(c)]
$\kappa$ is $\bP$-indestructibly measurable.
That is, there is a $\bP$-name $\utilde{\mathcal{D}}$ so that  $\forces_{\bP} ``\utilde{\mathcal{D}}$ is a normal ultrafilter on $\kappa$''.
\item[(d)]
For any $1\le n\le \bn$, $m<\om$, and $2\le k_{n,m}<\om$,
$\utilde{c}_{\,n,m}$ is a $\bP$-name for a function with domain
$A_{n,m}(\kappa)$ and range $k_{n,m}$.
\end{enumerate}

(B) is the statement:
There exist $(\utilde{g},\utilde{h})$ such that
\begin{enumerate}
\item[($\al$)]
\begin{enumerate}
\item[(i)]
$\utilde{g}$ is a $\bP$-name for an increasing function from $\kappa$ to $\kappa$.

\item[(ii)]
$\utilde{h}$ is  a $\bP$-name for a function from $\kgtwo$ into $\kgtwo$ mapping ${}^{\al}2$ into ${}^{\utilde{g}(\al)}2$, for each $\al<\kappa$.

\item[(iii)]
${\utilde{h}(\eta)}^{\frown}\lgl\ell\rgl \trngl \utilde{h}(\eta^{\frown}\lgl \ell\rgl)$, for $\eta\in \kgtwo$ and $\ell<2$.
\end{enumerate}

\item[($\beta$)] 
There is a $\bP$-name $\utilde{D}$ for a set in $\utilde{\mathcal{D}}$ such that
given  $1\le n\le \mathbf{n}$ and $m<\om$,
 for each
 $(u,\bar{\eta})\in A_{n,m}$  with $u\sse \utilde{D}$ and
$\eta^*_{\ell}\trngl\eta_\ell$ for all $\ell<\bf{n}$,
the value  of  $\utilde{c}_{\, n,m}(u,h''(\bar{\eta}))$ in the $\bP$-generic extension of $V$
depends only on $(n,m)$.
\end{enumerate}
\end{lem}

The proof  of Lemma \ref{lem.e8.AimpliesB} will use 
the following Theorem \ref{thm.ER} and Lemma \ref{lem.DH4.3}.

\begin{thm}[\Erdos--Rado, \cite{Erdos/Rado56}]\label{thm.ER}
For $r\ge 0$ finite and $\kappa$ an infinite cardinal,
$\beth_r(\kappa)^+\ra(\kappa^+)^{r+1}_{\kappa}$.
\end{thm}

\begin{defn}\label{defn.imagehom}
Let $\kappa<\lambda$  be given and let $\bP$ denote Cohen$(\kappa,\lambda)$.
We say that a subset $X\sse\bP$ is {\em image homogenized} if
\begin{enumerate}
\item[(a)]
All members of $X$ have domain with the same order-type: i.e., there is some 
$\zeta<\kappa$ such that 
 $\ot(\dom(p))=\zeta$ for all $p\in X$; and
\item[(b)]
For all $p,q\in X$
and $\xi<\zeta$,
if $\al<\lambda$ is the $\xi$-th element of $\dom(p)$ and $\beta<\lambda$ is the $\xi$-th element of $\dom(q)$,
then $p(\al)=q(\beta)$.
\end{enumerate}
\end{defn}

The following
 instance of 
Lemma 4.3 of \cite{Dobrinen/Hathaway16} 
allows us to assume only that  $\lambda\ge(\beth_{\mathbf{n}}(\kappa))^+$ in  (A)
of Lemma \ref{lem.e8.AimpliesB}.

\begin{lem}[\cite{Dobrinen/Hathaway16}]\label{lem.DH4.3}
Let  $1 \le n <\om$ be given
and  let $\kappa$ be a  strongly inaccessible cardinal satisfying 
$\kappa\ra(\mu_1)^{2n}_{\mu_2}$ for all $\mu_1,\mu_2<\kappa$.
Suppose that $\{p_{\vec{\alpha}}:\vec\alpha \in [\kappa]^n\}$
is an image homogenized set of conditions in 
the forcing $\bP= \mathrm{Cohen}(\kappa,\lambda)$, where $\lambda\ge \kappa$.
Then for each $\gamma < \kappa$ there is are sets 
$ K_i \subseteq \kappa$, $i< n$,
 such that
each 
$\ot(K_i)\ge\gamma$,
every element of
 $K_i$ is less than every element of $K_j$
whenever  $i < j < n$,
 and $\{p_{\vec{\alpha}}:\vec{\alpha}
 \in \prod_{i < n} K_i\}$ is a  pairwise compatible set of conditions.
\end{lem}

\noindent{\it Proof 
of Lemma \ref{lem.e8.AimpliesB}}.
Assume  (A) and,
without loss of generality, assume 
the conditions of 
 $\bP$ have  the following form:
For $p\in \bP$, 
$p$ is a
function 
 from  some subset of $\lambda \times \bn$ of cardinality less than $\kappa$
 into $\kgtwo$,
 the size of $\dom(p) \cap (\lambda \times \{\ell\})$ is the same for all $\ell< \mathbf{n}$,
 $p(\alpha,\ell)$ extends $\eta^*_{\ell}$
 for each $(\alpha,\ell) \in \dom(p)$,
 and 
all nodes in $\ran(p)$ have the same length.

Let $\utilde{G}$ denote  the canonical name for the
 generic object forced by $\bP$  over $V$, 
 and let  $V_1$ denote $V[G]$.
We let $A_n$ denote $A_n(\kappa)$ and  $A_{n,m}$ denote $A_{n,m}(\kappa)$, and 
note that    $A_n$ and $A_{n,m}$ remain the same in $V$ and $V_1$.
Given  $\vec\al\in [\kappa]^{\bn}$,
we write  $\vec\al$ in increasing order as
$\{\al_i:i<\bn\}$.
 In $V_1$,
 given
  $\vec\al\in[\kappa]^{\bn}$,
 $1\le n\le \bn$,
  and $\gamma<\kappa$,
 let
 \begin{equation}
 G_n(\vec\al)\re \gamma:=\lgl G(\al_i,i)\re\gamma:i<n\rgl.
 \end{equation}

\begin{claim}\label{claim.Claim1}
In $V_1$,
given $\vec\al\in[\lambda]^{\bn}$
there is an $X_{\vec\al}\in\mathcal{D}$ and 
an integer $k_{n,m,\vec\al}<k_{n,m}$,
for all $1\le n\le \bn$ and
$m<\om$,
such that
$c_{n,m}$ is constant on  the set
\begin{equation}
 A_{n,m,\vec\al}:=\{(u,G_n(\vec\al)\re\gamma) \in A_{n,m}: u\in[X_{\vec\al}]^m\mathrm{\ and\ } \gamma\in X_{\vec\al}\},
\end{equation}
with value $k_{n,m,\vec\al}$.
\end{claim}

\begin{proof}
Work in $V_1$ and 
fix $\vec{\al}\in[\lambda]^{\bn}$.
Given  $1\le n\le \bn$, notice that
 since $G_n(\vec\al)\re\gamma$ is completely determined by $\gamma$,
 for each $u\in[\kappa]^{<\om}$
 the function
 $c_{n,|u|}(u,\cdot)$  restricted to  the set
 $ \{G_n(\vec{\al})\re\gamma:\max(u)<\gamma<\kappa\}$
  is essentially a function on the ordinals in the interval $(\max(u),\kappa)$.
Since $\mathcal{D}$ is a normal ultrafilter on $\kappa$,
there is a set $X_{n,u}\in \mathcal{D}$ such that
$c_{n,|u|}(u,\cdot)$ is constant on
$\{G_n(\vec{\al})\re\gamma:\gamma\in X_{n,u}\}$, say with value $k_{n,u,\vec{\al}}$.
Without loss of generality, we may assume that $\max(u)<\min(X_{n,u})$.

Now for  each $\beta<\kappa$,
let
\begin{equation}
X_\beta=\bigcap\{X_{n,u}:1\le n\le \bn,\
u\in[\beta]^{<\om}\},
\end{equation}
and let $X=\Delta_{\beta<\kappa}X_\beta$.
Since $\mathcal{D}$ is  normal,  $X$ is in $\mathcal{D}$. 
Without loss of generality, we may further  assume that all ordinals in  $X$ are limit ordinals.

Let  $u\in[\kappa]^{<\om}$ and 
  $\gamma\in X$ be given with
$\max(u)<\gamma$.
Since all members of $X$ are limit ordinals, 
 $\max(u)+1<\gamma$;
let  $\beta_u=\max(u)+1$.
Note that 
 $\gamma\in X_{\beta_u}$ since  
$\beta_u<\gamma$.
It follows that  for each $1\le n\le\bn$, 
$\gamma\in X_{\beta_u}\sse X_{n,u}$, 
which implies that $c_{n,|u|}(u,G_n(\vec{\al})\re\gamma)=k_{n,u,\vec\al}$.

For each  $1\le n\le\bn$, define a coloring $\psi_n:[X]^{<\om}\ra k$ by setting $\psi_n(u)=k_{n,u,\vec\al}$.
Then there is a $Y_n\sse X$ in $\mathcal{D}$ such that for each $m<\om$,
$\psi_n$ is constant on $[Y_n]^m$; denote its value by $k_{n,m,\vec\al}$.
Letting  $X_{\vec\al}=\bigcap_{1\le n\le\bn}Y_n$,
we see that for each $u\in [X_{\vec\al}]^{<\om}$,
$k_{n,u,\vec\al}=k_{n,|u|,\vec\al}$.
 \end{proof}

In $V$,
for each  $1\le n\le\bn$, $m<\om$, and $\vec\al\in[\lambda]^{\bn}$,
there are $\bP$-names  $\utilde{X}_{\vec\al}$ and $\utilde{A}_{\, n,m,\vec\al}$ for the sets $X_{\vec\al}$ and  $A_{n,m,\vec\al}$ guaranteed by
 Claim \ref{claim.Claim1},
and a condition
 $p_{\vec\al}\in \bP$  which forces the following:
``$\utilde{X}_{\vec\al}\in\utcD$ and for all $1\le n\le \bn$ and $m<\om$,
$\utilde{c}_{\, n,m}$ takes value $k_{n,m, \vec\al}$ on
the set $\utilde{A}_{n,m, \vec\al}$.''
Without loss of generality, we may assume that  $p_{\vec\al}$ forces
``$\lh(\ran(p_{\vec\al}))\in \utilde{X}_{\vec\al}$
and
$\utilde{c}_{\,n,m}(u,\ran(p_{\vec\al}))=k_{n,m,\vec\al}$ for all $u\in [\utilde{X}_{\vec\al}]^m$ with $\max(u)<\lh(\ran(p_{\vec\al}))$.''

Now we find an image homogenized collection of  $p_{\vec\al}$'s.
For $\vec\al\in[\lambda]^{\bn}$,
recall that
$\dom(p_{\vec\al})$ is a subset of $\lambda\times \bn$ of cardinality less than $\kappa$.
Fix a bijection $b:\lambda\times \bn\ra\lambda$.
For $\vec\al\in[\lambda]^{\bn}$,
let  $b[\ran(p_{\vec\al})]$ 
denote the range of  $p_{\vec\al}$ ordered as 
the sequence
$\lgl p_{\vec\al}(b^{-1}(\beta)):\beta\in 
b[\dom(p_{\vec\al})]\rgl$.
Let $f$  be the coloring on $[\lambda]^{\bn}$ into $\kappa$ many colors defined as follows:
For  $\vec\al\in [\lambda]^{\bn}$,
\begin{align}
f(\vec\al)=
&\ {\lgl  \ot(b[\dom(p_{\vec\al})])\rgl}^{\frown}
{\lh(\ran(p_{\vec\al}))}^{\frown}
b[\ran(p_{\vec\al})]^{\frown}\cr
&\ {\lgl k_{n,m,\vec\al}:1\le n\le \bn,\ m<\om\rgl}^{\frown} \lgl p_{\vec\al}(\al_\ell,\ell):\ell<\bn\rgl.
\end{align}

Since $\lambda\ra(\kappa)^{\bn}_{\kappa}$,
there are $J\in[\lambda]^{\kappa}$,
 $\gamma^*,\delta^*<\kappa$,
nodes $\lgl t^*_{\beta}:\beta<\delta^*\rgl$ in ${}^{\kappa>}2$,
$k_{n,m}^*<k_{n,m}$,
and
$\vec{\zeta^*}\in {}^{\bn}({}^{\gamma^*}2)$
with $\eta^*_{\ell}\trngl \zeta^*_{\ell}$ for each $\ell<\bn$
such that
for all $\vec{\alpha} \in [J]^{\bn}$,
the following hold:

\begin{enumerate}
\item
 $\ot(b[\dom(p_{\vec{\alpha}})])=\delta^*$;
\item
 $\lh(\ran(p_{\vec\al}))=\gamma^*$;
\item
$b[\ran(p_{\vec\al})]=\lgl t^*_{\beta}:\beta<\delta^*\rgl$;
\item
$k_{n,m,\vec\al}=k_{n,m}^*$, for each $1\le n\le \bn$ and
$m<\om$;
\item
$\lgl p_{\vec\al}(\al_\ell,\ell):\ell<\bn\rgl=
\vec{\zeta^*}$.
\end{enumerate}


In particular, $|J|=\kappa$, and the set $\{p_{\vec\al}:\vec\al\in[J]^\bn\}$ is image homogenized. 
By Lemma \ref{lem.DH4.3},
for each $\gamma<\kappa$, 
there are sets $K_n\sse J$, $n<\bn$, 
with $K_0<K_1<\dots<K_{\bn-1}$ such that each $\ot(K_n)\ge\gamma$ and, letting $\vec{K}$ denote $\prod_{n<\bn} K_n$,
the set $\{p_{\vec\al}:\vec\al\in\vec{K}\}$ is pairwise compatible.

\begin{claim}\label{claim.2}
In $V_1$,
there is a $D\in\mathcal{D}$,
 a strictly increasing function  $g:\kappa\ra\kappa$ with
 $\ran(g)\sse D$,
and a strong subtree  $T\sse \kgtwo$
  with splitting levels in $\ran(g)$
  such that given
  $1\le n\le\bn$ and
 $m<\om$,
  for all $u\in[D]^m$ and
  $\bar{\eta}\sse T$  in $A_n$ with
   $\max(u)<\lg(\bar{\eta})$, we have
$$
c_{n,m}(u,\bar\eta)=k_{n,m}^*.
$$
\end{claim}

\begin{proof}
The function $g$ will be constructed recursively and will give the levels of the strong subtree $T\sse {}^{\kappa>}2$ which is being built. 
For ease of notation,
we will let $T(\gamma)$ denote  ${}^{g(\gamma)}2\cap T$.

For the base case, let $u=\emptyset$.
 Take $\vec{\al}$ to be any increasing sequence in  $J$.
By Claim \ref{claim.Claim1} and following exposition,
we
may take $g(0)=\lg(\ran(p_{\vec\al}))\in X_{\vec\al}$
 so that for all $1\le n\le \bn$,
\begin{equation}
 c_{n,0}(\emptyset,\lgl p_{\vec\al}(\al_\ell,\ell):\ell<n\rgl)
=k_{n,0}^*.
 \end{equation}
 Let  $T(0)=\{ p_{\vec\al}(\al_\ell,\ell):\ell<\bn\}$
  and let  $X_0=X_{\vec\al}$.

Given $0<\gamma<\kappa$,
suppose that for all $\delta<\gamma$,
$g(\delta)$ and $T(\delta)$  have been defined and satisfy the claim.
 If $\gamma$ is a limit ordinal, let $B$ denote
 the set of branches through
 $\bigcup_{\delta<\gamma} T(\delta)$.
 If $\gamma$ is a successor ordinal, let $B$ denote the set of immediate successors in $\kgtwo$
 of the branches through
 $\bigcup_{\delta<\gamma} T(\delta)$.
Take sets  $K_\ell\sse J$, 
$\ell<\bn$,
such that $K_\ell$ has 
 the same cardinality as the set of nodes
 in $B$ extending $\eta^*_{\ell}$,
 $K_0<\cdots <K_{\bn-1}$,
 and
$\{p_{\vec\al}:\vec\al\in\vec{K}\}$ is pairwise compatible, where 
 $\vec{K}=\prod_{\ell<\bn} K_{\ell}$.
Let  $q$ be a condition in $\bP$
such that,
for each $\ell<\bn$,
$\{q(\beta,\ell): \beta\in K_{\ell}\}$ is in one-to-one correspondence  with $\{\eta\in B:\eta^*_{\ell}\trngl \eta\}$.
In particular, $q$
 extends $\bigcup\{p_{\vec\al}:\vec\al\in \vec{K}\}$.

 Let
 \begin{equation}
 X_{\gamma}=
\bigcap\{X_{\delta}:\delta<\gamma\}\cap \bigcap\{X_{\vec\al}:\vec\al\in\vec{K}\,\},
 \end{equation}
 noting that this set is in $\mathcal{D}$.
 Extend $q$ to some $r\in\bP$ so that $\lg(\ran(r))\in X_{\gamma}$
 and, utilizing Claim \ref{claim.Claim1} again,
 for each $m<\om$, 
 $u\in  [\{g(\delta):\delta<\gamma\}]^{m}$,
 and
 $\vec\al\in \vec{K}$,
 $c_{n,m}(u,\cdot)$ has color $k_{n,m,\vec\al}=k_{n,m}^*$ on
 $\lgl r(\al_\ell,\ell):\ell<n\rgl$.
 We let $g(\gamma)=\lh(\ran(r))$ and $T(\gamma)=\{r(\al_\ell,\ell):\ell<\bn$ and $\vec\al\in\vec{K}\}$.

 In this way, we construct $g$ and $T$ and a decreasing sequence of sets  $X_{\gamma}\in\mathcal{D}$.
Let $D=\Delta_{\gamma<\kappa}X_{\gamma}$.
Each $g(\gamma)\in X_{\gamma}\setminus (\gamma+1)$, so $\ran(g)\sse X$.
Then for each $1\le n\le \bn$,
for any $\gamma<\kappa$
and
$u\in[\{ g(\delta):  \delta<\gamma\}]^{<\om}$,
 for each $\bar\eta\in A_n\cap T$ with
 $\lg(\bar\eta)\ge g(\gamma)$,
we have $c_{n,m}(u,\bar\eta)=k^*_{n,m}$.
 \end{proof}

Letting $\utilde{D}$ be a $\bP$-name for $D$, and letting $\utilde{h}$ be a $\bP$-name for the tree isomorphism from ${}^{<\kappa}2$ to $T$ 
 finishes the proof of  ($\beta$) of Lemma \ref{lem.e8.AimpliesB}.
\hfill$\square$

\begin{rem}
In fact, we get more than the Lemma claims:
The statement  $(\beta)$ actually
 holds for  all $(u,\bar{\eta})\in A_{n,m}$ with $u\sse \utilde{D}$ and $\bar{\eta}$ with  $\lg(\eta_\ell)\in \utilde{D}$  (rather than just $\lg(\eta_\ell)\in \ran(\utilde{g})$).
\end{rem}

We now restate Theorem \ref{thm.1beta2}
and prove it.

\begin{thm1.5}
Let
$1\le \bn<\om$ and $2\le k<\om$ be given.
Suppose
$\lambda\ge (\beth_{\bn}(\kappa))^+$ and that   $\kappa$ is measurable in the generic extension  via
$\bP=$
Cohen$(\kappa,\lambda)$  forcing.
 Let $\utilde{\bQ}$ be a $\bP$-name for Prikry forcing.
Then
$\HL_{n,k}(\kappa)$ holds for all  $1\le n\le \bn$ 
in the
generic extension  forced by $\bP * \utilde{\bQ}$.
\end{thm1.5}

\begin{proof}
Let
$1\le \bn<\om$ and $2\le k<\om$ be given.
Suppose
$\lambda\ge (\beth_{\bn}(\kappa))^+$ and that   $\kappa$ is measurable in the generic extension  via
$\bP=$
Cohen$(\kappa,\lambda)$  forcing.
Let $\utcD$ be a $\bP$-name in $V$ for a normal ultrafilter on $\check{\kappa}$.
Let $G$ be $\bP$-generic over $V$ and let $V_1$ denote $V[G]$.
In  $V_1$,  let $\bQ$ denote Prikry forcing with tails in $\mathcal{D}$.
Let  $H$ be
$\bQ$-generic over $V_1$, and let
 $V_2$ denote $ V_1[H]$.

Notice that    $A_n$ and $A_{n,m}$ remain the same in $V$, $V_1$, and $V_2$.
In $V_2$, for  each $1\le n\le \bn$,  let $c_n:A_n\ra k$ be a function.
In $V_1$,
 let $\utilde{c}_{\,n}$ be a $\bQ$-name for $c_n$.
For each
 $m<\om$,  
define in $V_1$
 a function $c_{n,m}:A_{n,m}\ra k$ by
\begin{equation}
 c_{n,m}(u,\bar{\eta})=j\Llra \exists X\in\mathcal{D}\
 (u,X)\forces_{\bQ} \utilde{c}_{\,n}(\bar{\eta})=j.
\end{equation}
Let $\utilde{c}_{\,n,m}$ be a $\bP$-name in $V$ for
$c_{n,m}$.
By Lemma \ref{lem.e8.AimpliesB},
there are $\bP$-names 
 $(\utilde{g},\utilde{h})$, $\utilde{T}$, 
and  $\utilde{D}$ in $V$ 
and  integers  $k^*_{n,m}$ 
 denoting the value in $V_1$ of 
$c_{n,m}(u,\bar{\eta})$ for 
$(u,\bar{\eta})\in A_{n,m}$ with 
$\bar{\eta}$ from $T$.

 \begin{claim}\label{claim.Claim3}
 In $V_1$,
there is an $E\sse D$ in $\mathcal{D}$ such that for each
$u\in[E]^{<\om}$ with $\max(u)<\al\in E$ and each $1\le n\le \bn$,
we have
$$
(u,E\setminus (\al+1))\forces_{\bQ} \utilde{c}_{\, n}(\bar\eta)=k^*_{n,|u|},
$$
for each $\bar\eta\in A_n\cap T$ such $\max(u)<\lg(\bar\eta)<\al$.
\end{claim}

\begin{proof}
 Given $\al<\kappa$ and 
 $u\in[\al]^{<\om}$,
let $Y_{\al,u}\sse D$ be in $\mathcal{D}$ so that
for all 
 $1\le n\le \bn$ and
$\bar\eta\in A_n\cap T$ with $\max(u)<\lg(\bar\eta)<\al$,
\begin{equation}\label{eq.uY}
(u,Y_{\al,u})\forces_{\bQ} \utilde{c}_{\, n}(\bar\eta)  =
c_n(u,\bar\eta).
\end{equation}
By Claim \ref{claim.2},
the value of $c_n(u,\eta)$ in equation (\ref{eq.uY}) is
 $k^*_{n, |u|}$.
Let $Y_{\al}=\bigcap\{Y_{\al,u}:u\in [\al]^{<\om}\}$,
and  let $E=\Delta_{\al<\kappa} Y_{\al}$.
Then for each $m<\om$,  whenever
$\al\in E$  and
$u\in[E\cap \al]^m$,
 we have  $E\setminus (\al+1)\sse Y_{\al,u}$.
 Thus,
$(u,E\setminus(\al+1))\forces_{\bQ} \utilde{c}_{\,n}(\bar\eta)=k^*_{n,m}$,
for all $\bar\eta$ in $T$ with $\lg(\bar\eta)<\al$.
\end{proof}

 In $V_2$, let $\bar{\lambda}=\lgl\lambda_i:i<\om\rgl$ be the generic Prikry generic sequence over $V_1$ given  by $H$.
By genericity, for each $Z\in\mathcal{D}$ in $V_1$,
 all but finitely many members of $\bar\lambda$ are contained in $Z$.
 In particular, $\bar\lambda\sse^* E$, so without loss of generality 
  we may assume
  $\bar\lambda\sse E$.
We may assume (by genericity) that the Prikry sequence has the property that
$|\lambda_{i+1}\setminus \lambda_i|$ is strictly increasing with limit $\kappa$.
Moreover,
we may assume that the sequence
$|\ran(g)\cap(\lambda_i,\lambda_{i+1})|$ is strictly increasing with limit $\kappa$.

 In $V_2$, we now construct a strong subtree $U$ of $T$
 with $\kappa$ many levels so that
 $c_n$ is constant  for each $1\le n\le \bn$.
By Claim \ref{claim.Claim3},
for all $1\le n\le\bn$ and all $\bar\eta\in A_n\cap T$ with $\lg(\bar\eta)<\lambda_0$,
we have
$(\emptyset, D\setminus( \lambda_0+1))\forces_{\bQ} \utilde{c}_{\, n}(\bar\eta)=k^*_{\emptyset, n}$.
In general, given $i<\om$,
for all $1\le n\le\bn$ and all $\bar\eta\in A_n\cap T$ with $\lg(\bar\eta)\in (\lambda_i,\lambda_{i+1})$,
we have
$(\{\lambda_0,\dots,\lambda_{i}\}, Y\setminus( \lambda_{i+1}+1))\forces_{\bQ} \utilde{c}_{\, n}(\bar\eta)=k^*_{n, i+1}$.
 Since each $k^*_{n,i}<k$, in $V_2$
 there is a set $I\in [\om]^{\om}$ such that
for all $1\le n\le\bn$,
 $c_n$ is constant for all $\bar\eta$  in $T$ with length in $\bigcup_{i\in I}(\lambda_i,\lambda_{i+1})$.
Then we may take a strong subtree $U$ of $T$ which has splitting levels in 
$\bigcup_{i\in I}(\lambda_i,\lambda_{i+1})$, so that $U$ has $\kappa$ many splitting levels.
Then this tree $U$ witnesses that HL$_{n,k}(\kappa)$ holds in $V_2$ for each $1\le n\le \bn$.
\end{proof}

\begin{rem}\label{rem.gamma}
If  the colorings  $c_n:A_n\ra k$ are in $V$,
then  the functions $\utilde{h}$ and $\utilde{g}$
in
conclusion of Lemma \ref{lem.e8.AimpliesB}
 can be found in $V$.
Hence, the strong subtree $T$ witnessing $\HL_{n,k}(\kappa)$ for  such $c_n$ in  Theorem \ref{thm.1beta2}
exists in $V$. 
\end{rem}

We now consider a version of Halpern--\Lauchli\ for infinite colorings.
For
 $2\le \theta_0\le \theta_1$,
  let 
$\HL_{n,(1,\theta_0,\theta_1)}[\kappa]$
 abbreviate the following statement:
Given a coloring
$c:\bigcup\{{}^n({}^{\al}2):\al<\kappa\}\ra \theta_1$,
then there is a suitable triple $(\kappa,T,A)$ and a  subset 
$u\sse \theta_1$ 
such that 
$ |u|<\theta_0$ and 
for all
$\al\in A$ and 
$\overline{\nu}=(\nu_0,\dots\nu_{n-1})$ with
$\eta_\ell^*\trngl \nu_\ell\in T\cap {}^{\al}2$
for each $\ell<n$,
then
$c(\overline{\nu})\in u$.

A minor  straightforward  modification of the proof of Theorem \ref{thm.1beta2}
yields the following theorem.

\begin{thm}\label{thm.2.8}
Under the assumptions of Theorem \ref{thm.1beta2},
if $\aleph_0<\theta<\kappa$,
we get that 
$\forces_{\bP*\utilde{\bQ}}
\HL_{n,(1,\aleph_1,\theta)}(\kappa)$.
\end{thm}


\section{Consistent Failures of Halpern--\Lauchli}

This section provides conditions under which various versions of  the Halpern--\Lauchli\ Theorem  fail. 
Our first theorem provides conditions which imply strong failure of Halpern--\Lauchli;  that is, failure of $\HL^+_{2,\theta}[\kappa]$ for all $2\le\theta<\kappa$.
For this, we will use 
  negative  square bracket partition relations.
Given  cardinals $1\le n<\om$ and $\theta,\mu\le \kappa$,
the square bracket partition relation
\begin{equation}
\kappa\ra [\mu]^n_{\theta}
\end{equation}
holds if
for every function $c:[\kappa]^n\ra\theta$,
there is a subset $A\sse \kappa$ with $|A|=\mu$ such that $\ran(c\re[A]^n)$ is a proper subset of $\theta$.
The negation
\begin{equation}
\kappa\nrightarrow [\mu]^n_{\theta}
\end{equation}
holds if there is a function $c:[\kappa]^n\ra\theta$ so that for each subset $A\sse \kappa$ with $|A|=\mu$,
 $\ran(c\re[A]^n)=\theta$.
The following lemma will aid in the proof of 
Theorem \ref{nsbpr->negHL}.

\begin{lem}\label{lem.nsbpr->negHL}
Suppose $\kappa$ is a strong limit cardinal and
either
\begin{enumerate}
\item[(a)]
$\kappa$ is strongly inaccessible, $\theta<\kappa$,
and
$\kappa\nrightarrow [\kappa]^2_{\theta}$; or
\item[(b)]
$\theta<\cf(\kappa)$ and
$\cf(\kappa)\nrightarrow [\cf(\kappa)]^2_{\theta}$.
\end{enumerate}
Then
$\HL^+_{2,\theta}[\kappa]$ fails.
\end{lem}

\begin{proof}
To prove (a),  suppose $\kappa$ is strongly inaccessible and  $\theta<\kappa$,  and
let $c:[\kappa]^2\ra\theta$ be a function witnessing
$\kappa\nrightarrow [\kappa]^2_{\theta}$.

\begin{claim}\label{claim.4}
For each $A\in[\kappa]^{\kappa}$, there is an $\al\in A$ such that  $c$ has range $\theta$ on the set
 $\{\{\al,\beta\}:\beta\in A\cap (\al,\kappa)\}$.
\end{claim}

\begin{proof}
Suppose not.
Then 
there is an $A\in[\kappa]^{\kappa}$ such that 
for each $\al\in A$, 
there is some  ordinal $e(\al)\in\theta$ which is not in 
the  range of $c$ on the set
$\{\{\al,\beta\}: \beta\in A\cap(\al,\kappa)\}$.
Since $\theta<\kappa$, there is an $A'\in [A]^{\kappa}$ such that $e$ is constant on $A'$.
But then $\ran(c\re [A']^2)\ne\theta$,  contradicting that $c$ witnesses 
$\kappa\ra[\kappa]^2_{\theta}$.
\end{proof}

Let  $\overline{<}$ be any sequence of well-orderings of
the levels of ${}^{\kappa>}2$.  
Define  the function
$d:\bigcup_{\al<\kappa}[{}^{\al}2]^2\ra \theta$  by
\begin{equation}
\{\nu_0,\nu_1\}\in [{}^{\al}2]^2\ \Lra\ 
d(\{\nu_0,\nu_1\})=c(\{\lg(\nu_0\cap\nu_1),\al\}).
\end{equation}
Let $(\kappa,T,A)$ be a suitable triple.
Take $A'\in [A]^{\kappa}$ such that between any two consecutive ordinals in $A'$, there are $\om$ many ordinals in $A$.
Fix $\al\in A'$ as in Claim \ref{claim.4}, and 
fix any node $\eta\in T\cap {}^{\al}2$ and distinct
nodes  $\eta_0,\eta_1 \in 
{}^{\al+1}2$
such that $\eta_0\cap\eta_1=\eta$.
For each
 $\beta\in A'\cap (\al,\kappa)$ there are at least $2^{\om}$ many pairs of distinct  nodes $\nu_0,\nu_1$ in $T\cap{}^{\beta}2$
such that $\eta_0\lhd \nu_0$ and 
$\eta_1\lhd\nu_1$
and $\nu_0<_{\beta}\nu_1$.
By Claim \ref{claim.4},
$d$ has range $\theta$ on the set 
\begin{equation}
\{\{\nu_0,\nu_1\}\in [T]^2:
\eta_0\lhd \nu_0,\ \eta_1\lhd\nu_1,\  \lg(\nu_0)=\lg(\nu_1)\in A',
\mathrm{\ and\ }
\nu_0<_{\lg(\nu_0)}\nu_1
\}.
\end{equation}
Thus,  $\HL^+_{2,\theta}[\kappa]$ fails.


The proof of  (b) is similar.
Suppose $\kappa$ is a strong limit cardinal,  let $\mu=\cf(\kappa)$, and let
 $c:[\mu]^2\ra\theta$ be a function witnessing that
$\mu\nrightarrow[\mu]^2_{\theta}$.
Let $\lgl \kappa_i:i<\mu\rgl$ be an increasing continuous sequence
with limit $\kappa$, and assume $\kappa_0=0$.
Define $h:\kappa\ra\mu$ by
\begin{equation}
h(\al)=i \Longleftrightarrow
\al\in[\kappa_i,\kappa_{i+1}).
\end{equation}
A proof similar to the one given for   Claim \ref{claim.4} yields the following:

\begin{claim}\label{claim.5}
Given $B\in[\mu]^{\mu}$,
there is an $a\in B$ such that $c$ has range $\theta$ on the set $\{\{a,b\}:b\in B\cap(a,\mu)\}$.
\end{claim}

Define the function 
 $d:\bigcup_{\al<\kappa}[{}^\al 2]^2\ra\theta$ by
\begin{equation}
\{\eta,\nu\} \in[{}^{\al}2]^2\Longrightarrow
d(\{\eta,\nu\})=c(\{h(\lg(\eta\cap\nu)),h(\al)\})
\end{equation}
if $\lg(\eta\cap\nu)$ and $\al$ are in different intervals of the partition $[\kappa_i,\kappa_{i+1})$  ($i<\mu$)  of $\kappa$,
and let $d(\{\eta,\nu\})=0$ otherwise.
Let $(\kappa, T,A)$ be a suitable triple.
Again take $A'\in[A]^{\kappa}$ such that between any two consecutive ordinals in $A'$, there are $\om$ many ordinals in $A$.

Let $B=h[A']$ and take $a\in B$ satisfying Claim \ref{claim.5}.
Let $\al\in A'$ be such that $h(\al)=a$,
and fix any node $\eta\in T\cap{}^{\al}2$
and
nodes  $\eta_0,\eta_1 \in 
{}^{\al+1}2$
such that $\eta_0\cap\eta_1=\eta$.
For each
 $\beta\in A'\cap (\al,\kappa)$ there are at least $2^{\om}$ many pairs of distinct  nodes $\nu_0,\nu_1$ in $T\cap{}^{\beta}2$
such that $\eta_0\lhd\nu_0$ and 
$\eta_1\lhd \nu_1$
and $\nu_0<_{\beta}\nu_1$.
By Claim \ref{claim.5},
$d$ has range $\theta$ on the set 
\begin{equation}
\{\{\nu_0,\nu_1\}\in[T]^2:
\eta_0\lhd \nu_0,\ \eta_1\lhd\nu_1,\ 
  \lg(\nu_0)=\lg(\nu_1)\in A',
\mathrm{\ and\ }
\nu_0<_{\lg(\nu_0)}\nu_1
\}.
\end{equation}
Thus,  $\HL^+_{2,\theta}[\kappa]$ fails.
\end{proof}

The next Lemma \ref{lem.Concl4.8} will also aid in the proof of  Theorem \ref{nsbpr->negHL}.
For this, we need the following.

\begin{defn}[\cite{Shelah_CardArithBK}, Definition 1.2 page 418]\label{defn}
 $\PR_1(\kappa,\mu,\sigma,\theta)$, where 
$\sigma+\theta\le \mu\le \kappa$ and
$\kappa$ is an infinite cardinal, means the following:
 There is a symmetric function $c:\kappa\times\kappa\ra\sigma$ such that
\begin{enumerate}
\item[$(*)$]
If $\xi<\theta$ and for all $i<\mu$, $\lgl \al_{i,\zeta}:\zeta<\xi\rgl$ is a strictly increasing sequence of ordinals less than $\kappa$ with 
the $\al_{i,\zeta}$ being distinct,
and if
 $\gamma<\sigma$,
then there are $i<j<\mu$ such that
\begin{enumerate}
\item[$\otimes$]\ \ 
$\zeta_1<\zeta_2<\xi\Lra c(\al_{i,\zeta_1},\al_{j,\zeta_2})=\gamma$.
\end{enumerate}
\end{enumerate}
\end{defn}

 $\PR_1(\kappa,\sigma,\theta)$ denotes  $\PR_1(\kappa,\kappa,\sigma,\theta)$, which implies $\kappa\nrightarrow[\kappa]^{\sigma}_{\theta}$.
The following lemma  is Conclusion 4.8  in  Chapter III Section 4 of \cite{Shelah_CardArithBK}:

\begin{lem}\label{lem.Concl4.8}
\begin{enumerate}
\item
Suppose  that  either 
\begin{enumerate}
\item[(a)]
 $\mu,\theta$ are regular cardinals and $\mu>\theta^+$, or 
\item[(b)]
$\mu$ is singular, $\mu<\aleph_{\mu}$, and $\theta=\cf(\mu)$.
\end{enumerate}
Then $\PR_1(\mu^+,\mu^+,\theta)$. 

\item
Suppose $\kappa>\aleph_0$ is inaccessible,
there is a stationary set
$S\sse\kappa$  which reflects in no inaccessible,
$\delta\in S$ implies $\cf(\delta)\ge\theta$,
and $\sigma<\kappa$ and $\aleph_0<\theta<\kappa$.
Then $\PR_1(\kappa,\sigma,\theta)$ holds.
\end{enumerate}
\end{lem}

Note that 
 $\PR_1(\kappa,\sigma,\theta)$ for any $2\le \theta$ implies
 $\PR_1(\kappa,\sigma,2)$, which  implies
 $\kappa \nrightarrow[\kappa]^2_\sigma$, which is what we will use in order to apply Lemma \ref{lem.nsbpr->negHL}.

\begin{thm1.6}
\begin{enumerate}
\item
If $\kappa$ is the first inaccessible, then $\HL^+_{2,\theta}[\kappa]$ fails, for each $2\le \theta<\kappa$.

\item
Suppose  $\kappa$ is inaccessible 
and not Mahlo.
Then  $\HL_{2,\theta}^+[\kappa]$ fails, for each $2\le \theta<\kappa$.

\item
If $\kappa$ is a singular  strong limit cardinal  and
 $\cf(\kappa)=\mu^+$, 
with $\mu$ regular,  
then
 $\HL^+_{2,\mu^+}[\kappa]$ fails.
\end{enumerate}
\end{thm1.6}

\begin{proof}
(1)
Suppose $\kappa$ is the first inaccessible and $2\le\theta<\kappa$, and let $\mu=\max(\theta,\aleph_1)$.
Then  the set
$S=\{\delta<\kappa:\cf(\delta)\ge \mu\}$
is stationary
and trivially does not 
 reflect in any inaccessible.
 Thus, by Lemma \ref{lem.Concl4.8} (2),
 Pr$_1(\kappa,\sigma,\mu)$ holds, for any $\sigma< \kappa$.
It follows that
 Pr$_1(\kappa,\theta, 2)$ holds,
 and hence, $[\kappa]\nrightarrow[\kappa]^2_\theta$ holds.
Then 
Lemma
\ref{lem.nsbpr->negHL} (a) implies that 
 $\HL^+_{2,\theta}[\kappa]$ fails.

(2)
Now suppose $\kappa$ is inaccessible 
and not Mahlo. 
By Lemma 4.1 (A) $\Rightarrow$ (C) in Chapter III of \cite{Shelah_CardArithBK},
 there is a stationary set
$S\sse \kappa$ which does not reflect in inaccessible cardinals,
and such that $\delta\in S$ implies $\cf(\delta)\ge\mu$, where $\mu=\max(\theta,\aleph_1)$.
By Lemma \ref{lem.Concl4.8} (2),
 Pr$_1(\kappa,\theta, 2)$ holds,
 and hence, $[\kappa]\nrightarrow[\kappa]^2_\theta$ holds.
Then $\HL^+_{2,\theta}[\kappa]$ fails, by
Lemma
\ref{lem.nsbpr->negHL} (a).

(3)
Lastly, suppose $\kappa$ is a singular strong limit cardinal, with 
$\cf(\kappa)=\mu^+$, where $\mu$ is regular.
By Lemma \ref{lem.Concl4.8} (1), $\PR_1(\mu^+,\mu^+,2)$ holds, and hence, $\mu^+\nrightarrow[\mu^+]^2_{\mu^+}$
holds. 
Then  $\HL^+_{2,\mu^+}(\kappa)$ fails,
by Lemma \ref{lem.nsbpr->negHL} (b).
\end{proof}


D\v{z}amonja, Larson, and Mitchell point out in Section 8 of  \cite{Dzamonja/Larson/MitchellQ09} that $\HL_{2,2}^+(\kappa)$ implies that $\kappa$ must be weakly compact.
Weakly compact cardinals are Mahlo and hence not the least strongly inaccessible. 
Theorem \ref{nsbpr->negHL}
 showed that if 
 $\HL^+_{2,\theta}[\kappa]$ holds for any $2\le\theta<\kappa$ where  $\kappa$ is strongly inaccessible,  then   $\kappa$ must be Mahlo.  Recall that 
 $\HL^+_{2,2}[\kappa]$ is exactly  $\HL^+_{2,2}(\kappa)$, leading to the following question.

\begin{question}\label{q.4.4}
Can 
$\HL^+_{2,\theta}[\kappa]$  consistently hold for some $2<\theta\le\kappa$ when $\kappa$ is a Mahlo, non-weakly compact cardinal?
\end{question}

We can only ask for consistency because Corollary 1.9 to the next theorem shows that
$\HL_{2,\theta}[\kappa]$  fails in $L$ for all $2\le\theta\le \kappa$ whenever $\kappa$ is strongly inaccessible and not weakly compact.


The following theorem shows  that the 
failure of   $\HL_{2,\theta}[\kappa]$, for all $2\le\theta\le\kappa$, 
follows 
from $\diamondsuit_S$ for a non-reflecting stationary subset $S\sse \kappa$.

\begin{thm1.7}
Assume $\kappa$ is strongly inaccessible,  $S\sse \kappa$ is a non-reflecting stationary set,
and $\diamondsuit_S$ holds.
Then $\HL_{2,\theta}[\kappa]$ fails, for each $2\le\theta\le\kappa$.
\end{thm1.7}

\begin{proof}
Let  $\kappa$ be inaccessible and  $S\sse \kappa$ be a non-reflecting stationary set, and suppose that
$\diamondsuit_S$ holds.
By possibly thinning  $S$,
we may assume that 
$S$ is a set of strong limit cardinals, and that 
there is a sequence
\begin{equation}\label{eq.diamondseq}
\{(\mu,T^0_\mu,T^1_\mu,A_\mu,F_\mu, \xi_\mu):\mu\in S\}
\end{equation}
such that  
\begin{enumerate}
\item[(a)]  
for each $\mu\in S$:
\end{enumerate}
\begin{enumerate}
\item[$\oplus_\mu$]
\begin{enumerate}
\item[$\bullet_1$]
For each $\ell<2$,
$(\mu,T^\ell_\mu,A_\mu)$ is a suitable triple,

\item[$\bullet_2$]
$F_\mu$ is a 1-1 function from $T^0_\mu$ onto $T^1_\mu$,

\item[$\bullet_3$]
$\eta\in T^0_\mu$ implies $\lg(F_\mu(\eta))=\lg(\eta)$,

\item[$\bullet_4$]
$(\ell<2\ \wedge\ \eta\in T^0_\eta\ \wedge\ \lg(\eta)\in A_\mu)$ implies $F_\mu(\eta^{\frown}\lgl\ell\rgl)=F_\mu(\eta)^{\frown}\lgl\ell\rgl$,

\item[$\bullet_5$]
$\eta\lhd\nu\in T^0_\mu$ implies $F_\mu(\eta)\lhd F_\mu(\nu)$,

\item[$\bullet_6$]
$\xi_\mu<\theta$,
\end{enumerate}
\end{enumerate}
 and such that 
\begin{enumerate}
\item[(b)]
if $(\kappa, T^0_\kappa, T^1_\kappa, A_\kappa,F_\kappa, \xi_\kappa)$ are as above,
i.e., $\oplus_\kappa$ holds,
then for stationarily many $\mu\in S$, we have
$$
(\mu,T^0_\mu,T^1_\mu,A_\mu,F_\mu, \xi_\mu)=
(\mu,T^0_\kappa\cap {}^{\mu>} 2,
T^1_\kappa\cap {}^{\mu>} 2,A_\kappa\cap\mu,F_\kappa\re T^0_\mu, \xi_\kappa).
$$
\end{enumerate}

Note that for each $\mu\in S$, $A_{\mu}$ is an unbounded subset of $\mu$, since
$(\mu,A_{\mu},T^\ell_{\mu})$ is suitable, for each $\ell<2$.

\begin{claim}\label{claim.3.1}
Given $\beta<\gamma<\kappa$
with $\beta\not\in S$,
then there is a  function $g=g_{\beta,\gamma}$ such that
\begin{enumerate}
\item
$g$ is one-to-one;
\item
$\dom(g)=S\cap\gamma\setminus(\beta+1)$;
\item
$\mu\in\dom(g)$ implies  $g(\mu)\in A_\mu\setminus(\beta+1)$.
\end{enumerate}
\end{claim}

\begin{proof}
The proof is by induction on $\gamma<\kappa$.
\vskip.1in

\noindent{\bf Case 1.}
$\gamma$ is a successor ordinal, say $\gamma=\gamma_1+1$.
Let $\beta\in \gamma\setminus S$ be given. 
If $\beta=\gamma_1$, then $S\cap\gamma\setminus (\beta+1)=\emptyset$ and the  empty function trivially satisfies (1)--(3); so 
now assume that $\beta<\gamma_1$. 
If  $\gamma_1\not\in S$, then
$S\cap\gamma\setminus(\beta+1)=S\cap\gamma_1\setminus(\beta+1)$ and we let $g_{\beta,\gamma}= g_{\beta,\gamma_1}$.

If $\gamma_1\in S$, since $A_{\gamma_1}$ is unbounded in $\gamma_1$ 
we can choose  $\beta_1\in A_{\gamma_1}\setminus (\beta+1)$.
By the induction hypothesis, we have
$g_1:=g_{\beta,\beta_1+1}$ on $S\cap (\beta_1+1)\setminus (\beta+1)$ satisfying (1)--(3)
and 
$g_2:=g_{\beta_1+1,\gamma_1}$ 
on $S\cap \gamma_1 \setminus (\beta_1+2)$  satisfying (1)--(3).
Then $\dom(g_1)\cap\dom(g_2)\sse (\beta,\beta_1]\cap (\beta_1+1, \gamma_1)=\emptyset$, so $g_1\cup g_2$ is a function.
Further. 
$\ran(g_1)\sse \beta_1$ and $\ran(g_2)\cap (\beta_1+2)=\emptyset$,
so $g_1\cup g_2$ is one-to-one and does not contain $\beta_1$ in its range. 
Define $g_{\beta,\gamma}=g_1\cup g_2\cup\{(\gamma_1,\beta_1)\}$.
Then $g_{\beta,\gamma}$ is a one-to-one function.
Moreover, 
\begin{align}
\dom(g_{\beta,\gamma}) &= \{\gamma_1\}
\cup
\big( S\cap ( (\beta, \beta_1]\cup (\beta_1+1,\gamma_1))
 \big)\cr
 & = S\cap (\beta,\gamma_1]\cr
& =S\cap  \gamma\setminus (\beta+1)
\end{align}
since $\gamma_1\in S$ and $\beta_1+1\not\in S$ as $S$ consists only of limit ordinals. 
As $\beta_1\in  A_{\gamma_1}\setminus (\beta+1)$, this along with (3) of the induction hypothesis for $g_1$ and $g_2$ imply that 
 $g_{\beta,\gamma}$  satisfies (1)--(3).
\vskip.1in

\noindent{\bf Case 2.}
$\gamma$ is a limit ordinal.
Let $\beta\in\gamma\setminus S$.
As $S$ does not reflect, $S\cap\gamma$ is not stationary in $\gamma$.
Recalling that $S$ consists only of limit ordinals, 
 there is an increasing continuous sequence $\lgl \beta_{\iota}:\iota<\cf(\gamma)\rgl$ of ordinals in $\gamma\setminus S$
such that $\gamma=\sup\{\beta_{\iota}:\iota<\cf(\gamma)\}$, with $\beta_0=\beta$.
Choose $g_{\iota}$ for $(\beta_{\iota},\beta_{\iota+1})$ according to the induction hypothesis, 
and let $g_{\beta,\gamma}=\bigcup\{g_{\iota}:\iota<\cf(\gamma)\}$.
Then $g_{\beta,\gamma}$ is a function satisfying (1)--(3).
\end{proof}

Using Claim \ref{claim.3.1}, we define a function
$c:[{}^{\al}2]^2\ra\theta$ for each $\al<\kappa$ as follows:
\begin{enumerate}
\item[$(*)_2$]
For $\gamma\in \ran(g_{0,\al})$ and $\mu$ such that $g_{0,\al}(\mu)=\gamma$,
if  $\{\nu_0,\nu_1\}\in [{}^{\al}2]^2$ satisfy $\gamma=\lg(F_{\mu}(\nu_0\re \gamma')\cap\nu_1)$, where $\gamma'$ is the least ordinal in $A_{\mu}$ above $\gamma$,
then define 
$c(\{\nu_0,\nu_1\})=\xi_\mu$;
otherwise,  $c(\{\nu_0,\nu_1\})=0$.
Let $c=\bigcup_{\al<\kappa}c_{\al}$.
\end{enumerate}

To finish the proof, 
let 
$\xi<\theta$
be given and let $T^{\ell}$, $\ell<2$, and $A\sse\kappa$ be such that 
$(\kappa, T^{\ell}, A)$ is a suitable triple for each $\ell<\kappa$, and let $F$ be the isomorphism from $T^0$ onto $T^1$.
Then  $\oplus_\kappa$ holds for the sequence
$(\kappa,T^0,T^1, A, F, \xi)$,
so  the set $S'$ of those $\mu\in S$ for which 
$$
(\mu,T^0_\mu,T^1_\mu,A_\mu,F_\mu,\xi_\mu)
=
(\mu, T^0\cap {}^{\mu>}2,T^1\cap {}^{\mu>}2,A\cap\mu, F\re T^0_\mu,\xi)
$$
holds is stationary.
Fix $\mu<\sigma$, both in $S'$.
Then  
$F_{\sigma}=F\re T^0_{\sigma}$
and  $F_{\sigma}\re T^0_{\mu}=F_{\mu}=F\re T^0_{\mu}$.
Choose $\al\in A_\sigma$ (which equals $A\cap \sigma$) such that $\mu+\om\le\al$.
This is possible since $\sigma\in S$ implies that $\sigma$ is a strong limit, and $A_\sigma$ is unbounded in $\sigma$.
Note that $T^{\ell}\cap {}^{\al}2= T^{\ell}_{\sigma}\cap {}^{\al}2$, for each $\ell<2$,

Let $\gamma=g_{0,\al}(\mu)$.
Then  $\gamma\in A_\mu=A\cap {}^{\mu}2$,
since $\mu\in S'$ and
 by (3) of Claim \ref{claim.3.1}.
Fix  a node $\eta\in T^1\cap {}^{\gamma}2$ and let $\eta_0=\eta^{\frown}0$ and $\eta_1=\eta^{\frown}1$.
Extend $\eta_0$ to some $\nu_0'$ in ${}^{\al}2\cap T^1_\sigma$,
and extend $\eta_1$ to some $\nu_1$ in 
${}^{\al}2\cap T^1_\sigma$.
Let $\nu_0=F_{\sigma}^{-1}(\nu_0')$.
Then $\nu_0$ is in $T^0\cap{}^{\al}2$.
Note  that, letting $\gamma'$ be the least ordinal in $A_{\mu}$ above $\gamma$, 
$$
F_{\mu}(\nu_0\re \gamma')
=F_{\sigma}(\nu_0\re \gamma')
=F_{\sigma}(\nu_0)\re\gamma'
=\nu'_0\re\gamma'.
$$
In particular, 
$\gamma=\lg (F_{\mu}(\nu_0\re \gamma')\cap\nu_1)$.
By the definition $(*)_2$ of $c$ it follows that
 $c(\{\nu_0,\nu_1\})=\xi_\mu$, which is $\xi$.
Since $\xi$ was an arbitrary ordinal less than $\theta$, we see that $\HL_{2,\theta}[\kappa]$ fails.
\end{proof}


Corollary \ref{cor.L} follows immediately.


\section{Open Problems}

We conclude by stating some of the multitude of open problems  regarding various versions of Halpern--\Lauchli\ at uncountable cardinals  and their consistency strengths.


\begin{question}
For $\kappa$ weakly compact,
if $\HL^+_{2,\theta}[\kappa]$  holds for some $2<\theta<\kappa$, then  must $\HL^+_{2,2}(\kappa)$ hold?
\end{question}

A similar question can be asked for Halpern-\Lauchli\ on products of  two  trees:

\begin{question}
Given $2<\theta<\kappa$, is  $\HL_{2,\theta}[\kappa]$   strictly weaker than $\HL_{2,\theta}(\kappa)$?
\end{question}

\begin{question}
For  $2\le n<\om$ and $2\le\theta<\kappa$, how do $\HL_{n,\theta}(\kappa)$ and $\HL^+_{n,\theta}[\kappa]$ compare? 
Are there models of ZFC where one holds but the other does not?
\end{question}

\bibliographystyle{amsplain}
\bibliography{references}

\end{document}